\documentclass[letterpaper,12pt]{article}

\usepackage[utf8]{inputenc}
\usepackage[dvipsnames]{xcolor}
\usepackage[utf8]{inputenc}
\usepackage{amssymb}
\usepackage{amsmath} 
\usepackage[colorinlistoftodos]{todonotes}
\usepackage{amsthm}
\usepackage[margin=1in]{geometry}
\usepackage{fancyhdr}
\usepackage{upgreek}
\usepackage{cite}
\usepackage{enumitem}
\numberwithin{equation}{section}

\usepackage[colorlinks]{hyperref}

\title{Fixed point results for contractions of polynomial type}

\author{Mohamed Jleli\footnote{Department of Mathematics, College of Science, King Saud University, Riyadh 11451, Saudi Arabia; E-mail address: jleli@ksu.edu.sa}, Cristina Maria P\u{a}curar\footnote{Faculty of Mathematics and Computer Science, Transilvania University of Bra\c{s}ov, 50 Iuliu Maniu Blvd., Bra\c{s}ov, Romania; E-mail address: cristina.pacurar@unitbv.ro},  Bessem Samet\footnote{Department of Mathematics, College of Science, King Saud University, Riyadh 11451, Saudi Arabia; E-mail address: bsamet@ksu.edu.sa}}

\date{}

\newtheorem{theorem}{Theorem}[section]

\theoremstyle{definition}
\newtheorem{definition}[theorem]{Definition}
\newtheorem{example}[theorem]{Example}
\newtheorem{proposition}[theorem]{Proposition}
\newtheorem{corollary}[theorem]{Corollary}

\theoremstyle{remark}
\newtheorem{remark}[theorem]{Remark}

\begin{document}

\setlength{\headheight}{13.59999pt}
\addtolength{\topmargin}{-1.59999pt}
\maketitle

\begin{abstract}
We introduce two new classes of single-valued  contractions of polynomial type defined on a metric space. For the first one, called the class of polynomial contractions, we establish two fixed point theorems. Namely, we first consider the case when the  mapping is continuous. Next, we weaken the continuity condition. In particular, we recover Banach's fixed point theorem. The second class, called the class of almost polynomial contractions, includes the class of almost contractions introduced by Berinde [Nonlinear Analysis Forum. 9(1) (2004) 43--53]. A fixed point theorem is established for almost polynomial contractions. The obtained result generalizes that derived by Berinde in the above reference. Several examples showing that our generalizations are significant, are provided.  

\vspace{0.1cm}

\noindent {\bf 2020 Mathematics Subject Classification:} 47H10, 54H25.\\
\noindent {\bf Key words and phrases:} Polynomial contractions, almost polynomial contractions, Picard-continuous mappings, weakly Picard operators,  fixed point.  
\end{abstract}

\section{Introduction}\label{sec1}

The most used techniques for studying the existence an uniqueness of solutions to nonlinear problems (for instance, integral equations, differential equations, fractional differential equations, evolution equations)  are based on the reduction of the problem to an equation of the form  $Tu=u$, where $T$ is a self-mapping defined on a certain set $X$ (usually equipped with a certain topology) and $u\in X$ is the unknown solution. Any solution $u$ to the previous equation is called a fixed point of $T$. So, the study of fixed point problems for different classes of mappings $T$ and different topological structures on $X$ is of great importance.

One of the most celebrated results in fixed point theory is the Banach fixed point theorem \cite{Banach}, which states that, if   $(X,d)$ is a complete metric space and $T: X\to X$ is a mapping  satisfying
\begin{equation}\label{Banach-CT}
d(Tw,Tz)\leq \lambda d(z,w)	
\end{equation}
for all $w,z\in X$, where $\lambda\in (0,1)$ is a constant, then 
\begin{itemize}
\item[($B_1$)] $T$ 	possesses one and only one fixed point;
\item[($B_2$)]For all $z_0\in X$, the sequence $\{z_n\}$ defined by $z_{n+1}=Tz_n$, converges to this fixed point. 
\end{itemize}
Any mapping $T: X\to X$ satisfying \eqref{Banach-CT} is called a contraction. The literature includes several generalizations and extensions of Banach's fixed point theorem.  Some of them were focused on weakening the right-side of inequality \eqref{Banach-CT}, see e.g. \cite{Berinde1,BO,Ciric,KH,RA,RE,Rus,Zhang}.  In other results, the underlying space is equipped with a generalized distance, see e.g. \cite{Branciari,CZ,DH,JS,Mustafa}. We also cite  the paper \cite{Nadler} of Nadler, who initiated the study of fixed points for 
multi-valued mappings. More recent fixed point results extending Banach's fixed point theorem can be found in \cite{AG,Berinde-20,Berinde4,PA2,Petrov,PE,Popescu,OZ}. 

On the other hand, despite the importance of Banach's fixed point theorem, this result is only concerned with continuous mappings. Namely, any mapping $T: X\to X$ satisfying \eqref{Banach-CT} is continuous on $(X,d)$. So, it is natural to ask whether it is possible to extend Banach's fixed point theorem to  mappings that are not necessarily continuous. The first work in this direction is due to  Kannan \cite{KAN}, where he introduced the class of mappings $T: X\to X$ satisfying the condition
$$
d(Tx,Ty)\leq \lambda \left[d(x,Tx)+d(y,Ty)\right]
$$
for all $x,y\in X$, where $\lambda\in \left(0,\frac{1}{2}\right)$ is a constant. Namely, it was shown that ($B_1$) and ($B_2$) hold also for the above class of mappings (when $(X,d)$ is a complete metric space). Following Kannan's result, several fixed point theorems have been  obtained without the requirement of the continuity of the mapping, see e.g. \cite{Berinde1,Berinde-08,Berinde3,CH,Ciric,RE}. In particular, Berinde \cite{Berinde1} introduced an interesting class of mappings, called the class of almost contractions (or weak contractions), which includes Kannan's mappings and many other classes of mappings. We recall below the definition of almost contractions. 

\begin{definition}\label{def1.1}
Let $(X,d)$ be a metric space. A mapping $T: X\to X$ is called an almost contraction, if there exists $\lambda\in (0,1)$ and $\ell>0$ such that 
\begin{equation}\label{contraction-de-Berinde}
d(Tx,Ty)\leq \lambda d(x,y)+\ell d(y,Tx)
\end{equation}
for every $x,y\in X$. 
\end{definition}
 
Berinde \cite{Berinde1} proved the following fixed point theorem for the above class of mappings. 

\begin{theorem}\label{T1.2}
Let $(X,d)$ be a complete metric space and $T: X\to X$ be an almost contraction. Then
\begin{itemize}
\item[{\rm{(i)}}] $T$ admits at least one fixed point;
\item[{\rm{(ii)}}] For all $z_0\in X$, the sequence $\{z_n\}$ defined by $z_{n+1}=Tz_n$, converges to a fixed point of $T$.  	
\end{itemize}
\end{theorem}
We point out that almost contractions can have more than one fixed point (see \cite[Example 1]{Berinde1}).

In this paper, we first introduce the class of polynomial contractions.  Two fixed point results are obtained for such mappings. Namely, we first consider the case when $T$ is continuous. Next, we weaken the continuity condition. Our obtained results recover Banach's fixed point theorem. Next, we introduce the class of almost polynomial contractions and establish a fixed point theorem for this class of mappings. Our obtained result generalizes Theorem \ref{T1.2}.
Several examples are provided to illustrate our results.  

\section{The class of polynomial contractions}\label{sec2}

We introduce below the class of  polynomial contractions.

\begin{definition}\label{def2.1}
Let $(X,d)$ be a metric space and $T: X\to X$	 be a given mapping. We say that $T$ is a polynomial contraction, if there exists $\lambda\in (0,1)$, a natural number $k\geq 1$ and a family of mappings  $a_i: X\times X\to [0,\infty)$, $i=0,\cdots,k$, such that 
\begin{equation}\label{SPC}
\sum_{i=0}^ka_i(Tx,Ty)d^i(Tx,Ty)\leq \lambda \sum_{i=0}^k a_i(x,y)d^i(x,y)	
\end{equation}
for every $x,y\in X$.  
\end{definition}

In this section, we are concerned with the study of fixed points for the above class of mappings.

We first consider the case when $T$ is a continuous mapping. 

\begin{theorem}\label{T2.2}
Let $(X,d)$ be a complete metric space and $T: X\to X$ be a polynomial contraction. Assume that the following conditions hold:
\begin{itemize}
\item[{\rm{(i)}}]$T$ is continuous;
\item[{\rm{(ii)}}] There exist $j\in\{1,\cdots,k\}$ and 	$A_j>0$ such that 
$$
a_j(x,y)\geq A_j,\quad x,y\in X.
$$
\end{itemize}
Then $T$ admits a unique fixed point $z^*\in X$. Moreover, for every $z_0\in X$, the Picard sequence  $\{z_n\}\subset X$ defined by $z_{n+1}=Tz_n$ for all $n\geq 0$, converges to $z^*$. 
\end{theorem}

\begin{proof}
We first prove that the set of fixed points of $T$ is nonempty.  Let $z_0\in X$ be fixed and  $\{z_n\}\subset X$ be the Picard sequence defined by  
$$
z_{n+1}=Tz_n,\quad  n\geq 0.
$$
Making use of \eqref{SPC} with $(x,y)=(z_0,z_1)$, we obtain
$$
\sum_{i=0}^ka_i(Tz_0,Tz_1)d^i(Tz_0,Tz_1)\leq \lambda \sum_{i=0}^k a_i(z_0,z_1)d^i(z_0,z_1),
$$
that is,
\begin{equation}\label{S1}
\sum_{i=0}^ka_i(z_1,z_2)d^i(z_1,z_2)\leq \lambda \sum_{i=0}^k a_i(z_0,z_1)d^i(z_0,z_1).
\end{equation}
Using again \eqref{SPC} with $(x,y)=(z_1,z_2)$, we obtain
$$
\sum_{i=0}^ka_i(Tz_1,Tz_2)d^i(Tz_1,Tz_2)\leq \lambda \sum_{i=0}^k a_i(z_1,z_2)d^i(z_1,z_2),
$$
that is,
$$
\sum_{i=0}^ka_i(z_2,z_3)d^i(z_2,z_3)\leq \lambda \sum_{i=0}^k a_i(z_1,z_2)d^i(z_1,z_2),
$$
which implies by \eqref{S1} that 
$$
\sum_{i=0}^ka_i(z_2,z_3)d^i(z_2,z_3)\leq \lambda^2 \sum_{i=0}^k a_i(z_0,z_1)d^i(z_0,z_1). 
$$
Continuing in the same way, we obtain by induction that   
\begin{equation}\label{S2}
\sum_{i=0}^ka_i(z_n,z_{n+1})d^i(z_n,z_{n+1})\leq \lambda^n \sum_{i=0}^k a_i(z_0,z_1)d^i(z_0,z_1),\quad n\geq 0.  	
\end{equation}
Since 
$$
a_j(z_n,z_{n+1})d^j(z_n,z_{n+1})\leq \sum_{i=0}^ka_i(z_n,z_{n+1})d^i(z_n,z_{n+1}),
$$
we obtain by (ii) that 
$$
A_j d^j(z_n,z_{n+1})\leq \sum_{i=0}^ka_i(z_n,z_{n+1})d^i(z_n,z_{n+1}),
$$
which implies by \eqref{S2} that 
\begin{equation}\label{S3}
d^j(z_n,z_{n+1})\leq  \lambda^n \sigma_{j,0},\quad n\geq 0,	
\end{equation}
where 
\begin{equation}\label{sigma0}
\sigma_{j,0}=A_j^{-1}\sum_{i=0}^k a_i(z_0,z_1)d^i(z_0,z_1).
\end{equation}
Then, making use of \eqref{S3} and the triangle inequality, we obtain that for all $n\geq 0$ and $m\geq 1$,
$$
\begin{aligned}
d^j(z_n,z_{n+m})&\leq d^j(z_n,z_{n+1})+d^j(z_{n+1},z_{n+2})+\cdots+d^j(z_{n+m-1},z_{n+m})\\
&\leq \sigma_{j,0}\left(	\lambda^n+\lambda^{n+1}+\cdots+\lambda^{n+m-1}\right)\\
&=\sigma_{j,0} \lambda^n \frac{1-\lambda^{m}}{1-\lambda}\\
&\leq \sigma_{j,0}  \frac{\lambda^n}{1-\lambda},
\end{aligned}
$$
which yields
$$
d(z_n,z_{n+m})\leq \left(\frac{\sigma_{j,0} }{1-\lambda}\right)^{\frac{1}{j}} (\lambda^{\frac{1}{j}})^n\to 0\mbox{ as }n,m\to \infty.
$$
This shows that $\{z_n\}$ is a Cauchy sequence. Since $(X,d)$ is complete, there exists $z^*\in X$ such that 
$$
\lim_{n\to \infty}d(z_n,z^*)=0,
$$
which implies by the continuity of $T$ that 
$$
\lim_{n\to \infty}d(z_{n+1},Tz^*)=\lim_{n\to \infty}d(Tz_n,Tz^*)=0.
$$
Then, from the uniqueness of the limit, we deduce that $Tz^*=z^*$, that is, $z^*$ is a fixed point of $T$. 

We now show that $z^*$ is the unique fixed point of $T$. Indeed, if $z^{**}\in X$ is another fixed point of $T$, i.e., $Tz^{**}=z^{**}$ and $d(z^*,z^{**})>0$, then making use of \eqref{SPC} with $(x,y)=(z^*,z^{**})$, we get 
$$
\sum_{i=0}^ka_i(Tz^*,Tz^{**})d^i(Tz^*,Tz^{**})\leq \lambda \sum_{i=0}^k a_i(z^*,z^{**})d^i(z^*,z^{**}),	
$$ 
that is,
\begin{equation}\label{S4}
\sum_{i=0}^k a_i(z^*,z^{**})d^i(z^*,z^{**})\leq \lambda \sum_{i=0}^k a_i(z^*,z^{**})d^i(z^*,z^{**}).
\end{equation}
On the other hand, from	(ii), we have 
$$
\begin{aligned}
\sum_{i=0}^k a_i(z^*,z^{**})d^i(z^*,z^{**}) &\geq  a_j(z^*,z^{**})d^j(z^*,z^{**})\\
&\geq A_j d^j(z^*,z^{**}).
\end{aligned}
$$
Since $A_j>0$ and $d(z^*,z^{**})>0$, we deduce that 
$$
\sum_{i=0}^k a_i(z^*,z^{**})d^i(z^*,z^{**})>0. 
$$
Then, dividing  \eqref{S4} by $\sum_{i=0}^k a_i(z^*,z^{**})d^i(z^*,z^{**})$, we reach a contradiction with $\lambda\in (0,1)$. Consequently, $z^*$ is the unique fixed point of $T$. This completes the proof of Theorem \ref{T2.2}. 
\end{proof}

We now study some particular cases of Theorem \ref{T2.2}. 

\begin{proposition}\label{PR2.3}
Let $(X,d)$ be a metric space and $T: X\to X$ be a	polynomial contraction (in the sense of Definition \ref{def2.1}). Assume that the following conditions hold:
\begin{itemize}
\item[(i)] $a_0\equiv 0$, i.e., $a_0(x,y)=0$ for all $x,y\in X$;
\item[(ii)]For all $i\in\{1,\cdots,k\}$, there exists 	$B_i>0$ such that 
$$
a_i(x,y)\leq B_i,\quad x,y\in X;
$$
\item[(iii)] There exist $j\in\{1,\cdots,k\}$ and 	$A_j>0$ such that 
$$
a_j(x,y)\geq A_j,\quad x,y\in X.
$$
\end{itemize}
Then $T$ is continuous. 
\end{proposition}

\begin{proof}
Let $\{u_n\}\subset X$ be a sequence such that 
\begin{equation}\label{SP1}
\lim_{n\to \infty}d(u_n,u)=0,	
\end{equation}
for some $u\in X$. 	Using (i) and making use of \eqref{SPC} with $(x,y)=(u_n,u)$, we get 
$$
\sum_{i=1}^ka_i(Tu_n,Tu)d^i(Tu_n,Tu)\leq \lambda \sum_{i=1}^k a_i(u_n,u)d^i(u_n,u),\quad n\geq 0,
$$
which implies by (ii) and (iii) that 
\begin{equation}\label{SP2}
A_jd^j(Tu_n,Tu)\leq \lambda \sum_{i=1}^k B_id^i(u_n,u),\quad n\geq 0.
\end{equation}
Then, making use of \eqref{SP1} and passing to the limit as $n\to \infty$ in \eqref{SP2}, we get
$$
\lim_{n\to \infty}d^j(Tu_n,Tu)=0,
$$ 
which is equivalent to 
$$
\lim_{n\to \infty}d(Tu_n,Tu)=0.
$$  
This shows that $T$ is a continuous mapping. 
\end{proof}

From Theorem \ref{T2.2} and Proposition \ref{PR2.3}, we deduce the following result. 

\begin{corollary}\label{CR2.4}
Let $(X,d)$ be a complete metric space and $T: X\to X$ be a polynomial contraction. Assume that the following conditions hold:
\begin{itemize}
\item[(i)] $a_0\equiv 0$;
\item[{\rm{(ii)}}] For all $i\in\{1,\cdots,k\}$, there exists 	$B_i>0$ such that 
$$
a_i(x,y)\leq B_i,\quad x,y\in X;
$$
\item[{\rm{(iii)}}] There exist $j\in\{1,\cdots,k\}$ and 	$A_j>0$ such that 
$$
a_j(x,y)\geq A_j,\quad x,y\in X.
$$
\end{itemize}
Then $T$ admits a unique fixed point $z^*\in X$. Moreover, for every $z_0\in X$, the Picard sequence  $\{z_n\}\subset X$ defined by $z_{n+1}=Tz_n$ for all $n\geq 0$, converges to $z^*$. 
\end{corollary}

The following result is an immediate consequence of Corollary \ref{CR2.4}.

\begin{corollary}\label{CR2.5}
Let $(X,d)$ be a complete metric space and $T: X\to X$ be a given mapping. Assume that there exist $\lambda\in (0,1)$, a natural number $k\geq 1$ and a finite sequence $\{a_i\}_{i=1}^k\subset (0,\infty)$ such that 
\begin{equation}\label{CT-CR2.5}
\sum_{i=1}^ka_id^i(Tx,Ty)\leq \lambda \sum_{i=1}^k a_id^i(x,y)
\end{equation}
for every $x,y\in X$. Then  $T$ admits a unique fixed point $z^*\in X$. Moreover, for every $z_0\in X$, the Picard sequence  $\{z_n\}\subset X$ defined by $z_{n+1}=Tz_n$ for all $n\geq 0$, converges to $z^*$.   	
\end{corollary}

\begin{remark}
Observe that Corollary \ref{CR2.5} recovers Banach's fixed point theorem. 	Indeed, taking $k=1$ and $a_1=1$, \eqref{CT-CR2.5} reduces to 
$$
d(Tx,Ty)\leq \lambda d(x,y),\quad x,y\in X. 
$$
\end{remark}

We provide below an example to illustrate Theorem \ref{T2.2}.

\begin{example}\label{ex2.7}
Let $X=\{x_1,x_2,x_3,x_4\}$ and $T: X\to X$ be the mapping defined by 
$$
Tx_1=x_1,\,\, Tx_2=x_3,\,\, Tx_3=x_4,\,\, Tx_4=x_1. 
$$ 
Let $d$ be the discrete metric on $X$, i.e., 
$$
d(x_i,x_j)=\left\{\begin{array}{llll}
1 &\mbox{if}& i\neq j,\\[4pt]
0  &\mbox{if}& i=j. 	
\end{array}
\right.
$$
Consider the mapping $a_0: X\times X\to [0,\infty)$ defined by 
\begin{eqnarray*}
&&a_0(x_i,x_j)=a_0(x_j,x_i),\\
&&a_0(x_i,x_i)=0,\\
&&a_0(x_1,x_2)=a_0(x_2,x_3)=3,\\
&&a_0(x_1,x_3)=a_0(x_3,x_4)=2,\\
&& a_0(x_1,x_4)=1,\\
&&a_0(x_2,x_4)=6.  
\end{eqnarray*}
We claim that 
\begin{equation}\label{claim1}
a_0(Tx,Ty)+d(Tx,Ty)\leq \frac{3}{4} \left(a_0(x,y)+d(x,y)\right)	
\end{equation}
for every $x,y\in X$, that is, $T$ is a polynomial contraction in the sense of Definition \ref{def2.1} with $k=1$, $a_1\equiv 1$ and $\lambda=\frac{3}{4}$. If $x=y$ or $(x,y)=(x_1,x_4)$, then \eqref{claim1} is obvious. Then, by symmetry, we have just to show that \eqref{claim1} holds for all $x_i,x_j\in X$ with $1\leq i<j\leq 4$ and $(i,j)\neq (1,4)$. Table \ref{Tb} provides the different values of $a_0(Tx_i,Tx_j)+d(Tx_i,Tx_j)$ and  $a_0(x_i,x_j)+d(x_i,x_j)$ for all  $1\leq i<j\leq 4$, which confirm \eqref{claim1}. 
\begin{table}[htt!]
\begin{center}
\begin{tabular}{|c | c | c |}
  \hline			
  $(i,j)$ & $a_0(Tx_i,Tx_j)+d(Tx_i,Tx_j)$  & $a_0(x_i,x_j)+d(x_i,x_j)$ \\
  \hline
  $(1,2)$ & 3 & 4 \\
  \hline
  $(1,3)$ & 2 & 3 \\
  \hline 
 $(2,3)$ & 3 & 4 \\
  \hline 
  $(2,4)$ & 3 & 7 \\
  \hline 
  $(3,4)$ & 2 & 3 \\
  \hline 
\end{tabular}
\caption{The values of $a_0(Tx_i,Tx_j)+d(Tx_i,Tx_j)$ \& $a_0(x_i,x_j)+d(x_i,x_j)$}\label{Tb}
\end{center}
\end{table}
Then, all the conditions of Theorem \ref{T2.2} are satisfied ((ii) is satisfied with $A_1=1$). On the other hand, $T$ admits a unique fixed point $z^*=x_1$, which confirms our obtained result. 

Remark that Banach's fixed point theorem is not applicable in this example.  Indeed, we have
$$
\frac{d(Tx_1,Tx_2)}{d(x_1,x_2)}=\frac{d(x_1,x_3)}{d(x_1,x_2)}=1. 
$$

We also notice that, if we consider the mapping 
$$
\mathcal{D}(x,y)=d(x,y)+a_0(x,y),\quad x,y\in X, 
$$
then \eqref{claim1} reduces to 
$$
\mathcal{D}(Tx,Ty)\leq \frac{3}{4} \mathcal{D}(x,y),\quad x,y\in X. 
$$
However, $\mathcal{D}$ is not a metric on $X$. This can be easily seen observing that (see Table \ref{Tb}) 
$$
\mathcal{D}(x_2,x_4)=7>6=\mathcal{D}(x_2,x_1)+\mathcal{D}(x_1,x_4). 
$$
\end{example}

We now weaken the continuity condition imposed on $T$ in Theorem \ref{T2.2}. 

\begin{definition}\label{ddef-2.8}
Let $(X,d)$ be a metric space. A mapping  $T: X\to X$ is called Picard-continuous, if for all $z,w\in X$, we have 
$$
\lim_{n\to \infty}d(T^nz,w)=0\implies \lim_{n\to \infty} d(T(T^nz),Tw)=0, 
$$ 
where $T^0z=z$ and $T^{n+1}z=T(T^nz)$ for all $n\geq 0$. 
\end{definition}

Remark that, if $T: X\to X$ is continuous, then $T$ is Picard-continuous. However, the converse is not true. The following example shows this fact.

\begin{example}\label{ex2.9}
Let $X=[a,b]$, where $a,b\in \mathbb{R}$ and $a<b$. Consider the mapping  $T: X\to X$  defined by 
$$
Tx=\left\{\begin{array}{llll}
a &\mbox{if}& a\leq  x< b,\\[4pt]
\frac{a+b}{2} 	 &\mbox{if}& x=b.
\end{array}
\right.
$$	
Let $d$ be the standard metric on $X$, that is, $d(x,y)=|x-y|$ for all $x,y\in X$.  Clearly, the mapping $T$ is not continuous at $b$. However, $T$ is Picard-continuous in the sense of Definition \ref{ddef-2.8}. Indeed,  observe  that for all $z\in X$, we have 
$$
T^nz=a\quad \mbox{ for all }n\geq 2. 
$$
So, if for some $z,w\in X$, we have 
$$
\lim_{n\to \infty}d(T^nz,w)=0,
$$
then $w=a$ and 
$$
\lim_{n\to \infty}d(T(T^nz),w)=\lim_{n\to\infty} d(Ta,w)=d(Ta,w)=0, 
$$
which shows that $T$ is Picard-continuous.
\end{example}

\begin{theorem}\label{T2.10}
Let $(X,d)$ be a complete metric space and $T: X\to X$ be a polynomial contraction. Assume that the following conditions hold:
\begin{itemize}
\item[{\rm{(i)}}]$T$ is Picard-continuous;
\item[{\rm{(ii)}}] There exist $j\in\{1,\cdots,k\}$ and 	$A_j>0$ such that 
$$
a_j(x,y)\geq A_j,\quad x,y\in X.
$$
\end{itemize}
Then $T$ admits a unique fixed point $z^*\in X$. Moreover, for every $z_0\in X$, the Picard sequence  $\{z_n\}\subset X$ defined by $z_{n+1}=Tz_n$ for all $n\geq 0$, converges to $z^*$. 
\end{theorem}

\begin{proof}
We first prove that the set of fixed points of $T$ is nonempty. Let $z_0\in X$ be fixed and  $\{z_n\}\subset X$ be the Picard sequence defined by  
$$
z_{n+1}=Tz_n,\quad  n\geq 0,
$$	
that is,
$$
z_n=T^nz_0,\quad n\geq 0. 
$$
From the proof of Theorem \ref{T2.2}, we know that $\{z_n\}$ is a Cauchy sequence, which implies by the completeness of $(X,d)$ that there exists $z^*\in X$ such that 
$$
\lim_{n\to \infty}d(T^nz_0,z^*)=0. 
$$
Then, by the Picard continuity of $T$, it holds that 
$$
\lim_{n\to \infty}d(T^{n+1}z_0,Tz^*)=\lim_{n\to \infty}d(T(T^nz_0),Tz^*)=0, 
$$
which implies by the uniqueness of the limit that $z^*$ is a fixed point of $T$. The rest of the proof is similar to that of Theorem \ref{T2.2}.
\end{proof}

We now provide an example to illustrate Theorem \ref{T2.10}. 

\begin{example}\label{ex2.10}
Let $X=[0,1]$ and $T: X\to X$ be the mapping defined by 
$$
Tx=\left\{\begin{array}{llll}
\frac{1}{4} &\mbox{if}& 0\leq x<1,\\[4pt]
0 &\mbox{if}& x=1.	
\end{array}
\right.
$$	
Let $d$ be the standard metric on $X$, i.e., 
$$
d(x,y)=|x-y|,\quad x,y\in X. 
$$
Remark that $T$ is not a continuous mapping, but it is Picard-continuous (see Example \ref{ex2.9}).   

Consider now the mapping $a_0: X\times X\to [0,\infty)$ defined by 
$$
a_0(x,y)=\frac{5}{6}\left(x\left|x-\frac{1}{4}\right|+y\left|y-\frac{1}{4}\right|\right),\quad x,y\in X. 
$$
We claim that 
\begin{equation}\label{claim2}
a_0(Tx,Ty)+d(Tx,Ty)\leq \frac{1}{2} \left(a_0(x,y)+d(x,y)\right)	
\end{equation}
for all $x,y\in X$, that is, $T$ is a polynomial contraction in the sense of Definition \ref{def2.1} with $k=1$, $a_1\equiv 1$ and $\lambda=\frac{1}{2}$.  We discuss three possible cases (due to the symmetry of $a_0$). \\
Case 1: $x,y\in [0,1)$.  In this case, we have
$$
\begin{aligned}
a_0(Tx,Ty)+d(Tx,Ty)&=a_0\left(\frac{1}{4},\frac{1}{4}\right)+d\left(\frac{1}{4},\frac{1}{4}\right)\\
&=0,	
\end{aligned}
$$
which yields \eqref{claim2}. \\
Case 2: $x \in [0,1)$ and $y=1$.  In this case, we have
$$
\begin{aligned}
a_0(Tx,Ty)+d(Tx,Ty)&=a_0\left(\frac{1}{4},0\right)+d\left(\frac{1}{4},0\right)\\
&=\frac{1}{4}	
\end{aligned}
$$
and
$$
\begin{aligned}
\frac{a_0(x,y)+d(x,y)}{2}&=\frac{a_0\left(x,1\right)+d\left(x,1\right)}{2}\\
&=\frac{5}{12}x\left|x-\frac{1}{4}\right|+\frac{5}{16}+\frac{|x-1|}{2}\\
&\geq \frac{1}{4},
\end{aligned}
$$
which yields \eqref{claim2}.\\
Case 3: $x=y=1$. In this case, we have
$$
\begin{aligned}
a_0(Tx,Ty)+d(Tx,Ty)&=a_0(0,0)+d\left(0,0\right)\\
&=0,	
\end{aligned}
$$
which yields \eqref{claim2}.

Therefore, \eqref{claim2} holds. Consequently, Theorem \ref{T2.10} applies. On the other hand, $z^*=\frac{1}{4}$ is the unique fixed point of $T$, which confirms the obtained result given by Theorem \ref{T2.10}. 

Notice that in this example, Theorem \ref{T2.2} is inapplicable since $T$ is not  continuous.  
\end{example}

\section{The class of almost polynomial contractions}\label{sec3}

Motivated by Berinde \cite{Berinde1},  we introduce below the class of almost polynomial contractions. 

\begin{definition}\label{def3.1}
Let $(X,d)$ be a metric space and $T: X\to X$	 be a given mapping. We say that $T$ is an almost polynomial contraction,  if there exists $\lambda\in (0,1)$, a natural number $k\geq 1$, a finite sequence $\{L_i\}_{i=0}^k\subset (0,\infty)$  and a family of mappings  $a_i: X\times X\to [0,\infty)$, $i=0,\cdots,k$, such that 
\begin{equation}\label{SPC-B}
\sum_{i=0}^ka_i(Tx,Ty)d^i(Tx,Ty)\leq \lambda \sum_{i=0}^k a_i(x,y)\left[d^i(x,y)+L_id^i(y,Tx)\right]	
\end{equation}
for every $x,y\in X$.  
\end{definition}

We recall below the concept of weakly Picard operators, which was introduced by Rus (see e.g. \cite{Rus1,Rus2,Rus3}).

\begin{definition}
Let $(X,d)$ be a metric space and $T: X\to X$	 be a given mapping. We say that $T$ is a weakly Picard operator, if 
\begin{itemize}
\item[(i)] The set of fixed points of $T$ is nonempty;
\item[(ii)] For all $z_0\in X$, the Picard sequence $\{T^nz_0\}$ is convergent and its limit belongs to the set of fixed points of $T$. 
\end{itemize}
\end{definition}

Our main result in this section is the following fixed point theorem.

\begin{theorem}\label{T3.3}
Let $(X,d)$ be a complete metric space and $T: X\to X$ be an almost polynomial contraction. Assume that the following conditions hold:
\begin{itemize}
\item[{\rm{(i)}}]$T$ is Picard-continuous;
\item[{\rm{(ii)}}] There exist $j\in\{1,\cdots,k\}$ and 	$A_j>0$ such that 
$$
a_j(x,y)\geq A_j,\quad x,y\in X.
$$
\end{itemize}
Then $T$ is a weakly Picard operator. 
\end{theorem}

\begin{proof}
Let $z_0\in X$ be fixed and  $\{z_n\}\subset X$ be the Picard sequence defined by  
$$
z_{n+1}=Tz_n,\quad  n\geq 0.
$$
Making use of \eqref{SPC-B} with $(x,y)=(z_0,z_1)$, we obtain
$$
\sum_{i=0}^ka_i(Tz_0,Tz_1)d^i(Tz_0,Tz_1)\leq \lambda \sum_{i=0}^k a_i(z_0,z_1)\left[d^i(z_0,z_1)+L_id^i(z_1,Tz_0)\right],			
$$
that is,
\begin{equation}\label{SSS1}
\sum_{i=0}^ka_i(z_1,z_2)d^i(z_1,z_2)\leq \lambda \sum_{i=0}^k a_i(z_0,z_1)d^i(z_0,z_1).
\end{equation}
Again, making use of \eqref{SPC-B} with $(x,y)=(z_1,z_2)$, we obtain
$$
\sum_{i=0}^ka_i(Tz_1,Tz_2)d^i(Tz_1,Tz_2)\leq \lambda \sum_{i=0}^k a_i(z_1,z_2)\left[d^i(z_1,z_2)+L_id^i(z_2,Tz_1)\right],			
$$
that is,
$$
\sum_{i=0}^ka_i(z_2,z_3)d^i(z_2,z_3)\leq \lambda \sum_{i=0}^k a_i(z_1,z_2)d^i(z_1,z_2),
$$
which gives us thanks to \eqref{SSS1} that 
$$
\sum_{i=0}^ka_i(z_2,z_3)d^i(z_2,z_3)\leq \lambda^2 \sum_{i=0}^k a_i(z_0,z_1)d^i(z_0,z_1).
$$
Continuing this process, we get by induction that 
$$
\sum_{i=0}^ka_i(z_n,z_{n+1})d^i(z_n,z_{n+1})\leq \lambda^n \sum_{i=0}^k a_i(z_0,z_1)d^i(z_0,z_1),\quad n\geq 0,
$$
which implies from (ii) that 
$$
d^j(z_n,z_{n+1})\leq  \lambda^n \sigma_{j,0},\quad n\geq 0,	
$$
where $\sigma_{j,0}$ is given by \eqref{sigma0}. Next, proceeding as in the proof of Theorem \ref{T2.2}, we obtain that $\{z_n\}$ is a Cauchy sequence, which implies by the completeness of $(X,d)$ the existence of $z^*\in X$ such that 
$$
\lim_{n\to \infty}d(z_n,z^*)=0. 
$$
Finally, taking into consideration that $T$ is Picard-continuous, we get 
$$
\lim_{n\to \infty}d(z_{n+1},Tz^*)=0,
$$
which implies by the uniqueness of the limit that $z^*=Tz^*$. The proof of Theorem \ref{T3.3} is then completed. 
\end{proof}

We now investigate some special cases of Theorem \ref{T3.3}. 

\begin{proposition}\label{PP}
Let $(X,d)$ be a metric space and $T: X\to X$ be a given mapping. Assume that there exist $\lambda\in (0,1)$, a natural number $k\geq 1$ and two finite sequence $\{a_i\}_{i=1}^k, \{L_i\}_{i=1}^k\subset (0,\infty)$ such that 
\begin{equation}\label{CT-CR3.3}
\sum_{i=1}^ka_id^i(Tx,Ty)\leq \lambda \sum_{i=1}^k a_i\left[d^i(x,y)+L_id^i(y,Tx)\right]
\end{equation}
for every $x,y\in X$. Then $T$ is Picard-continuous  	
\end{proposition}

\begin{proof}
Let $z,u\in X$ be such that 
\begin{equation}\label{OK1}
\lim_{q\to \infty}  d(T^{q}z,u)=0.	
\end{equation}
Making use of \eqref{CT-CR3.3} with $(x,y)=(T^{q}z,u)$, we obtain
$$
\sum_{i=1}^ka_id^i(T(T^{q}z),Tu)\leq \lambda \sum_{i=1}^k a_i\left[d^i(T^{q}z,u)+L_id^i(u,T(T^{q}z))\right],
$$
that is,
$$
\sum_{i=1}^ka_id^i(T^{q+1}z,Tu)\leq \lambda \sum_{i=1}^k a_i\left[d^i(T^{q}z,u)+L_id^i(u,T^{q+1}z)\right], 
$$
which implies that 
$$
d(T(T^{q}z),Tu)\leq \frac{\lambda}{a_1} \sum_{i=1}^k a_i\left[d^i(T^{q}z,u)+L_id^i(u,T^{q+1}z)\right].
$$
Then, passing to the limit as $q\to \infty$ in the above inequality and making use of \eqref{OK1}, we obtain
$$
\lim_{q\to \infty}  d(T(T^{q}z),Tu)=0,	
$$
which proves that $T$ is Picard-continuous. 	
\end{proof}

From Theorem \ref{T3.3} and Proposition \ref{PP}, we deduce the following result. 

\begin{corollary}\label{CR3.4}
Let $(X,d)$ be a complete metric space and $T: X\to X$ be a given mapping. Assume that there exist $\lambda\in (0,1)$, a natural number $k\geq 1$ and two finite sequence $\{a_i\}_{i=1}^k, \{L_i\}_{i=1}^k\subset (0,\infty)$ such that \eqref{CT-CR3.3} holds for every $x,y\in X$. Then $T$ is a weakly Picard operator. 
\end{corollary}

\begin{proof}
Remark that \eqref{CT-CR3.3} is a special case of \eqref{SPC-B} with $a_0\equiv 0$ and $a_i$ is constant for all $i\in\{1,\cdots,k\}$. Then, by 	Proposition \ref{PP}, Theorem \ref{T3.3} applies. 
\end{proof}

\begin{remark}
Taking $k=1$, $a_1=1$ and $L_1=\frac{\ell}{\lambda}$, where $\ell>0$, \eqref{CT-CR3.3}	reduces to	\eqref{contraction-de-Berinde}. Then, by Corollary \ref{CR3.4}, we recover Berinde's fixed point theorem (Theorem \ref{T1.2}). 
\end{remark}

We now give some examples to illustrate the above obtained results. The following example shows that under the conditions of Theorem \ref{T3.3}, we may have more than one fixed point. 

\begin{example}\label{example-3.6}
Let $X=\left\{x_1,x_2,x_3\right\}$ and $T: X\to X$ be the mapping defined by 
$$
Tx_1=x_1,\,\, Tx_2=x_2,\,\, Tx_3=x_1.
$$
The set $X$ is equipped with the discrete metric 
$$
d(x_i,x_j)=\left\{\begin{array}{llll}
1 &\mbox{if}& i\neq j,\\[4pt]
0  &\mbox{if}& i=j. 	
\end{array}
\right.
$$
We claim that the mapping $T$ satisfies 	\eqref{CT-CR3.3} with $k=2$, $a_1=a_2=1$, $\lambda=\frac{2}{3}$ and $L_1=L_2=\frac{1}{2}$, that is,
\begin{equation}\label{ClaimOK}
d(Tx,Ty)+d^2(Tx,Ty)\leq \frac{2}{3}\left[d(x,y)+\frac{1}{2}d(y,Tx)+d^2(x,y)+\frac{1}{2}d^2(y,Tx)\right]
\end{equation}
for all $x,y\in X$. 
\begin{table}[htt!]
\begin{center}
\begin{tabular}{|c | c | c |}
  \hline			
  $(i,j)$ & $d(Tx_i,Tx_j)+d^2(Tx_i,Tx_j)$  & $d(x_i,x_j)+\frac{1}{2}d(x_j,Tx_i)+d^2(x_i,x_j)+\frac{1}{2}d^2(x_j,Tx_i)$ \\
  \hline
  $(1,2)$ & 2 & 3 \\
  \hline
  $(2,1)$ & 2 & 3 \\
  \hline 
   $(2,3)$ & 2 & 3 \\
  \hline 
  $(3,2)$ & 2 & 3 \\
  \hline 
\end{tabular}
\caption{The values of $d(Tx_i,Tx_j)+d^2(Tx_i,Tx_j)$ \& $d(x_i,x_j)+\frac{1}{2}d(x_j,Tx_i)+d^2(x_i,x_j)+\frac{1}{2}d^2(x_j,Tx_i)$}\label{Tb2}
\end{center}
\end{table} 
If $x=y$ or $(x,y)\in\{(x_1,x_3),(x_3,x_1)\}$,  then \eqref{ClaimOK} is obvious. Table \ref{Tb2} gives the different values of $d(Tx_i,Tx_j)+d^2(Tx_i,Tx_j)$ and $d(x_i,x_j)+\frac{1}{2}d(x_j,Tx_i)+d^2(x_i,x_j)+\frac{1}{2}d^2(x_j,Tx_i)$ for all $i,j\in\{1,2,3\}$ with $i\neq j$ and $(i,j)\not\in\{ (1,3),(3,1)\}$, which confirm \eqref{ClaimOK}.  Then Corollary \ref{CR3.4} applies. On the other hand, the set of fixed points of $T$ is nonempty, which confirms the obtained result provided by Corollary \ref{CR3.4}. Remark that the set of fixed points of $T$ is equal to $\{x_1,x_2\}$. 
\end{example}

An  other example that illustrates    Theorem \ref{T3.3} is   given below. 

\begin{example}\label{ex3.7}
Let $X=[0,1]$ and $T: X\to X$ be the mapping defined by
$$
Tx=\left\{\begin{array}{llll}
\frac{1}{4} &\mbox{if}& 0\leq x<1,\\[4pt]
0 &\mbox{if}& x=1.	
\end{array}
\right.
$$	
Let $d$ be the standard metric on $X$, i.e., 
$$
d(x,y)=|x-y|,\quad x,y\in X. 
$$ 	
We recall that $T$ is Picard-continuous (see Example \ref{ex2.9}).  

Consider  the mapping $a_0: X\times X\to [0,\infty)$ defined by 
$$
a_0(x,y)=\left|4x^2-3x+\frac{1}{2}\right|+\left|4y^2-3y+\frac{1}{2}\right|,\quad x,y\in X. 
$$
We claim that 
\begin{equation}\label{claim-Fex}
a_0(Tx,Ty)+d(Tx,Ty)\leq \frac{1}{2} \left(2a_0(x,y)+d(x,y)+d(y,Tx)\right)
\end{equation}
for all $x,y\in X$, that is, $T$ is an almost polynomial contraction in the sense of Definition \ref{def3.1} with $k=1$, $a_1\equiv 1$, $L_0=L_1=1$ and $\lambda=\frac{1}{2}$. We discuss four possible cases.\\
Case 1: $0\leq x,y<1$. In this case, we have 
$$
\begin{aligned}
a_0(Tx,Ty)+d(Tx,Ty)&=a_0\left(\frac{1}{4},\frac{1}{4}\right)+d\left(\frac{1}{4},\frac{1}{4}\right)\\
&=0.
\end{aligned}
$$  
Then \eqref{claim-Fex} holds. \\
Case 2: $0\leq x<1$, $y=1$. In this case, we have
$$
\begin{aligned}
a_0(Tx,Ty)+d(Tx,Ty)&=a_0\left(\frac{1}{4},0\right)+d\left(\frac{1}{4},0\right)\\
&=\frac{1}{2}+\frac{1}{4}\\
&=\frac{3}{4}\\
&=d(y,Tx)\\
&\leq \frac{1}{2} \left(2a_0(x,y)+d(x,y)+d(y,Tx)\right).
\end{aligned}
$$ 
Then \eqref{claim-Fex} holds. \\
Case 3: $x=1$, $0\leq y<1$. In this case, we have
$$
\begin{aligned}
a_0(Tx,Ty)+d(Tx,Ty)&=a_0\left(0,\frac{1}{4}\right)+d\left(0,\frac{1}{4}\right)\\
&=\frac{1}{2}+\frac{1}{4}\\
&=\frac{3}{4}
\end{aligned}
$$ 
and
$$
\begin{aligned}
a_0(x,y)&=a_0(1,y)\\
&=\frac{3}{2}+\left|4y^2-3y+\frac{1}{2}\right|\\
&\geq \frac{3}{2}.
\end{aligned}
$$
Therefore, it holds that 
$$
\begin{aligned}
a_0(Tx,Ty)+d(Tx,Ty) &=\frac{3}{4}\\
&\leq \frac{3}{2}\\
&\leq \frac{1}{2} \left[2 a_0(x,y)\right]\\
&\leq \frac{1}{2} \left(2a_0(x,y)+d(x,y)+d(y,Tx)\right).	
\end{aligned}
$$
Then \eqref{claim-Fex} holds.\\
Case 4: $x=y=1$. In this case, we have
$$
\begin{aligned}
a_0(Tx,Ty)+d(Tx,Ty)&=a_0(0,0)+d(1,1)\\
&=1\\
&=d(y,Tx)\\
&\leq \frac{1}{2} \left(2a_0(x,y)+d(x,y)+d(y,Tx)\right).
\end{aligned}
$$ 
Then \eqref{claim-Fex} holds. 

Consequently,  \eqref{claim-Fex} is satisfied for all $x,y\in X$. Notice also that condition (ii) of Theorem \ref{T3.3} holds with $j=1$ and $A_1=1$. Then Theorem \ref{T3.3} applies. Observe that $z^*=\frac{1}{4}$ is a fixed point of $T$, which confirms the result given by Theorem \ref{T3.3}. 
\end{example}

\subsection*{Data availability statement}
No new data were created or analyzed in this study.

\subsection*{Conflict of interest}
The authors declare that they have no competing interests.

\subsection*{Funding}
The third author  is supported by Researchers Supporting Project number (RSP2024R4), King Saud University, Riyadh, Saudi Arabia.

\subsection*{Authors' contributions}
All authors contributed equally and significantly in writing this article. All authors read and approved the final manuscript.

\end{document}